\documentclass[12pt]{amsart}
\usepackage{amsmath,amsthm,amssymb}
\usepackage{amsfonts}
\usepackage[english]{babel}
\usepackage[dvips]{graphicx}
\usepackage[active]{srcltx}
\usepackage[applemac]{inputenc}
\usepackage{float}
\usepackage{url}
\usepackage[a4paper=true,%
breaklinks=true,%
colorlinks=true,%
linkcolor={blue},%
citecolor={red},%
pdftitle={A glimpse inside the mathematical kitchen},%
pdfauthor={J. Arias de Reyna, J. van de Lune},%
pdfkeywords={monotonicity},%
pdfstartview=FitH,
]{hyperref}

\newtheorem{theorem}{Theorem}[section]
\newtheorem{proposition}[theorem]{Proposition}
\newtheorem{lemma}[theorem]{Lemma}

\theoremstyle{definition}

\theoremstyle{remark}

\newcommand\N{\mathbf{N}}

\newcommand\R{\mathbf{R}}   

\renewcommand\Re{\mathrm{Re\,}}

\newcommand\medio{{\textstyle \frac{1}{2}}\,}

\newfont{\cmbsy}{cmbsy10}
\newfont{\cmmib}{cmmib10}

\begin{document}

\title{A glimpse inside the mathematical kitchen}

\author{Juan Arias de Reyna}
\address{Facultad de Matemáticas,
Universidad de Sevilla,
 Apdo.\-~1160, 41080-Sevilla, Spain.}
\email{arias@us.es}

\author{Jan van de Lune}
\address{Langebuorren 49, 9074 CH Hallum\\
The Netherlands (formerly at CWI, Amsterdam)}
\email{j.vandelune@hccnet.nl}

\thanks{The first author is supported by  grant
MTM2009-08934.}


\begin{abstract}
We prove the inequality
\begin{displaymath}
\sum_{k=1}^\infty (-1)^{k+1}\frac{r^k\cos k\phi}{k+2}
                < \sum_{k=1}^\infty(-1)^{k+1}\frac{ r^k}{k+2}
\end{displaymath}
for $0 < r \le 1$ and  $0 < \phi < \pi$.
 
For the case  $r = 1$  we give two proofs.
The first one is by means of a general numerical technique 
( maximal slope principle ) for proving inequalities between 
elementary functions. The second proof is fully analytical.
Finally we prove a general rearrangement theorem and apply it to the
remaining case $0 < r < 1$.
 
Some of these inequalities are needed for obtaining general sharp 
bounds for the errors committed when applying the Riemann-Siegel 
expansion of Riemann's zeta function.
\end{abstract}

\date{\today}

\maketitle


\section{The  problem to be dealt with in this note.}

The main goal of this note is to prove that for $0<r\le 1$ and
$0<\varphi<\pi$
\begin{equation}\label{main}
\frac{r\cos\varphi}{3}-\frac{r^2\cos2\varphi}{4}+\frac{r^3\cos3\varphi}{5}-+\cdots<
\frac{r}{3}-\frac{r^2}{4}+\frac{r^3}{5}-+\cdots
\end{equation}
We soon recognized that this is not a trivial problem, and
still hold that view.

\section{Motivation.}

In one of our studies \cite{A} of the  error(s), inherent in
using the Riemann-Siegel formula for the Riemann $\zeta$
function ( see, for example, Edwards \cite{E} or Gabcke
\cite{G} ), we encountered the following problem: Find a sharp
bound of the integral
\begin{equation}\label{intmotivation}
\int_C(1-z)^{-\sigma}e^{-x^2
f(z)}\frac{dz}{z^{k+1}}\quad\text{where}\quad
f(z):=-\frac{\log(1-z)}{z^2}-\frac{1}{z}-\frac{1}{2}.
\end{equation}
Here $k$ is a natural number, $\sigma$ and $x$ denote arbitrary real
numbers, $C$ is a simple circular contour around $z=0$ with radius
$r\in(0,1]$, and $\log(1-z)$ is the principal logarithm:
$\log(1-z):=-\sum_{k=1}^\infty \frac{z^k}{k}$ for $|z|\le 1$,
$z\ne1$.

The usual technical paper proceeds, as directly as possible, to the
final result. However, it occurred to us that an interested reader
might appreciate a glimpse inside the mathematical kitchen. To this
end, our note will provide the reader a detailed summary of the
struggles we encountered along the way to our final solution.

\section{Reduction of the problem.}

We soon recognized that our problem concerning the integral in
\eqref{intmotivation} may be reduced to finding a suitable sharp
upper bound of $- \Re f(z)$ for $| z | = r$,  i.~e., a suitable
sharp upper bound of $- \Re f(re^{i\varphi})$ for
$-\pi<\varphi<\pi$.

It is easily seen that $\Re f(re^{i\varphi})$ is an even function of
$\varphi$, so that we may restrict ourselves to  $0 \le
\varphi<\pi$. The reader may know that in such cases we have a habit
of first making a Plot ($\,$using Mathematica$\,$) of the function(s) in
question.

\begin{figure}[H]
  \includegraphics[width=8cm]{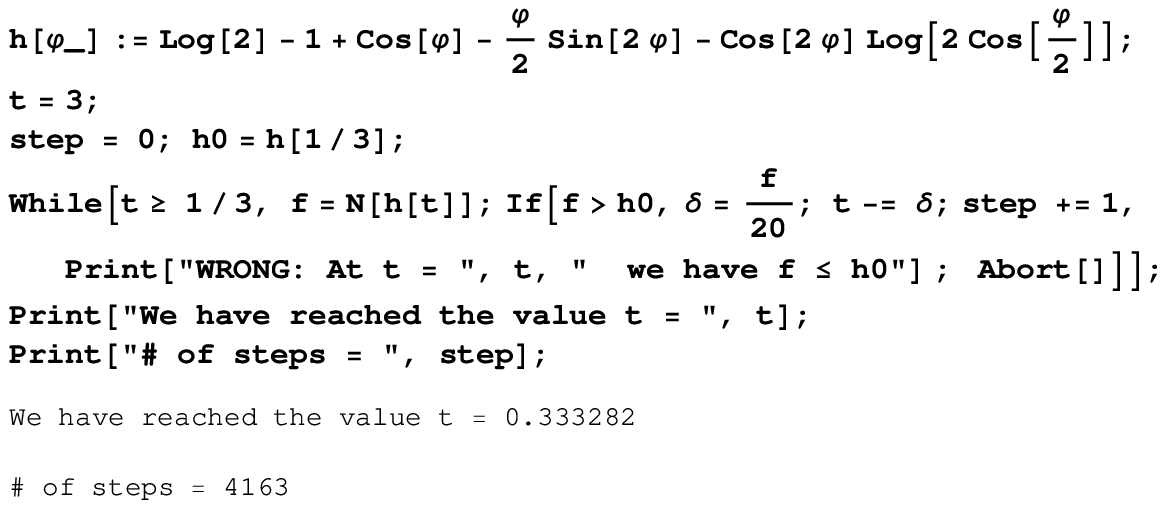}
  \includegraphics[width=6.8cm]{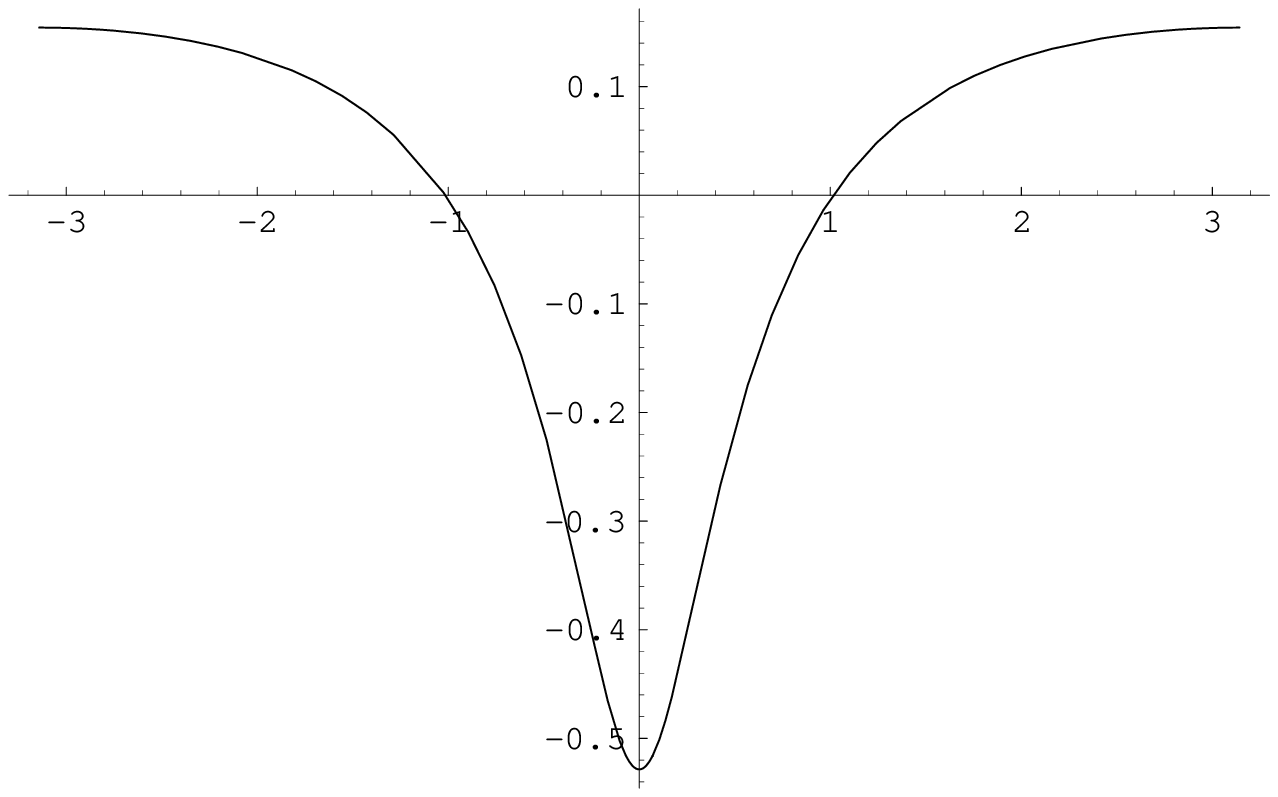}\\
  \label{F:Xray}
\end{figure}

After having made various such plots of $- \Re f(re^{i\varphi})$  we
decided to be satisfied with showing that  $- \Re f(re^{i\varphi})$
is maximal for $\varphi=\pi$ or, equivalently, that
\begin{equation}\label{first}
-\Re f(re^{i\varphi})<-\Re f(re^{\pi i})=-\Re f(-r)\quad\text{for
all}\quad 0<\varphi<\pi.
\end{equation}

( Although in \cite{A} this inequality was actually needed only for
$r=1$, $r=8/9$ and $r=0.883$, we are striving for some generality
here. )

Using the power series expansion of $\log(1-z)$ we may write
\eqref{first} as \eqref{main}. As said before, proving inequality
\eqref{main} will be our main goal in this note. (~There are no
serious convergence problems in \eqref{main}.~)

\section{Application of the Maximal Slope Principle.}

Suppose we have a differentiable real function  $h(x)$  on an
interval $[a , b]$  with  $| h'(x) | \le M$  for all  $x \in[a
, b]$. ( Here we assume  $M > 0$, because otherwise we are not
dealing with a serious problem. ) As a simple application of
the Mean Value Theorem, the Maximal Slope Principle ( MSP ) now
asserts the following : If, for example, $h(b)> 0$  then $h(x)$
is also positive for all $x \in(x_1, b)$ where $x_1 := \max ( a
, b - \frac{h(b)}{M})$. ( Just draw a picture$\;$! )

Note that if $x_1>a$ and $h(x_1)>0$ we may repeat this procedure (
until we reach an $x_1\le a$ ).

\subsection{Some Kitchen Prep Work.}
Of the many useful applications of the MSP we briefly mention a few
examples:
\begin{itemize}

\item{} Flett's function $F(t):=\sum_{n=1}^\infty
    \frac{\sin(t/n)}{n}$ has no zeroes in the $t$-interval
    $(0,48)$. The first zero is found at $t=48.418454\dots$

\item{} For all $n\in[2,10]$ the function
$Q(x):=\frac{1^x+2^x+3^x+\cdots+n^x+(n+1)^x}{1^x+2^x+3^x+\cdots+n^x}$
is log-convex  \text{( in $x$ )} on the entire real line $\R$.

(To this we might add our conjecture that $Q(x)$  is
log-convex \text{( in $x$ )} on $\R$ for all $n\in\N$. )

\item{} By means of the MSP one may prove ( or disprove )
    excruciatingly complicated inequalities $L < R$ where
    $L$ and $R$ are exponential polynomials.

\item{} The MSP may also be used to locate zeroes of real
    functions such as, for example,
    $R_{23}(t):=\sum_{n=1}^{23}\frac{\cos(t\log n)}{n}$.
\end{itemize}

\subsection{Application of the MSP method.}

Following in the footsteps of Hilbert and Pólya, we apply the MSP to
the function $ -\Re f(re^{i\varphi})$ for the simplest case  $r=1$.
In \cite[pp. 126-127.]{FJ} we read : 'Courant describes Hilbert's
method of dealing with problems as follows: \emph{He was a most
concrete, intuitive mathematician who invented, and very consciously
used, a principle: namely, if you want to solve a problem first
strip the problem of everything that is not essential. Simplify it,
specialize it as much as you can without sacrificing its core. Thus
it becomes simple, as simple as it can be made, without losing any
of its punch, and then you solve it. The generalization is a
triviality, which you do not need to pay too much attention to. This
principle of Hilbert's proved extremely useful for him and also for
others who learned it from him; unfortunately it has been
forgotten.}'.

In the present case ( $r=1$ )  we thus have to show that  $ - \Re
f(e^{i\varphi}) < - \Re f(-1)$ for $0 < \varphi <\pi$. It is clear
that in this inequality we may replace $\varphi$  by  $\pi-\varphi$,
so that we may just as well prove that

\begin{equation}
-\Re f(-e^{-i\varphi})<-\Re f(-1)\qquad \text{for } 0<\varphi<\pi.
\end{equation}

Writing
\begin{equation}\label{defu}
u(\varphi):=-\Re f(-e^{-i\varphi})=\Re\left(e^{2i\varphi
}\log(1+e^{-i\varphi })-e^{i\varphi }+\medio\right)
\end{equation}
we may also write our inequality  as $u(\varphi)<u(0)$.

We have ( for $-\pi<\varphi<\pi$ )
\begin{align*}
u(\varphi)&=\Re\Bigl[(e^{2i\varphi}\log\bigl(e^{-i\varphi/2}(e^{i\varphi/2}+e^{-i\varphi/2})\bigr)
-e^{i\varphi}+\frac12\Bigr])\\
&=\Re\Bigl[(\cos2\varphi+i\sin2\varphi)\Bigl(
\log\Bigl(2\cos\frac{\varphi}{2}\Bigr)-\frac{i\varphi}{2}\Bigr)
-e^{i\varphi}+\frac12\Bigr]\cr &=
\cos(2\varphi)\log\Bigl(2\cos\frac{\varphi}{2}\Bigr)+
\frac{\varphi}{2}\sin(2\varphi)-\cos\varphi+\frac12.
\end{align*}

Now we define
\begin{align}
h(\varphi):&=u(0)-u(\varphi)=\label{defh}\\
&=\log2-1+\cos\varphi-\frac{\varphi}{2}\sin(2\varphi)
-\cos(2\varphi)\log\Bigl(2\cos\frac{\varphi}{2}\Bigr).\label{formulah}
\end{align}
We have just seen that we have to show that $h(\varphi)>0$ for
$0<\varphi<\pi$.

Before applying the MSP to $h(\varphi)$ we first show that
$h(\varphi)$ is positive on the intervals
$0<\varphi\le\frac{1}{3}$ and $3\le \varphi<\pi$.

\begin{lemma}
$h(\varphi)>0$ for $0<\varphi\le\frac{1}{3}$.
\end{lemma}

\begin{proof}
We will use the elementary inequalities
\begin{displaymath}
1-\frac{t^2}{2!}<\cos
t<1-\frac{t^2}{2!}+\frac{t^4}{4!}\quad\text{and}\quad \sin t <
t -\frac{t^3}{3!}+\frac{t^5}{5!}.
\end{displaymath}
Then, with $x=\frac{\varphi^2}{8}-\frac{\varphi^4}{384}$ we have
$\cos\frac{\varphi}{2}<1-x$ so that
\begin{equation}
\frac{1}{\cos\frac{\varphi}{2}}>\frac{1}{1-x}>1+x
\end{equation}
and
\begin{align*}
h(\varphi)&=\log2-1+\cos\varphi-\frac{\varphi}{2}\sin(2\varphi)
-\cos(2\varphi)\log\Bigl(2\cos\frac{\varphi}{2}\Bigr)>\\
&>\Bigl(2\varphi^2-\frac{2}{3}\varphi^4\Bigr)\log2+(1-2\varphi^2)\log(1+x)-
\frac{\varphi}{2}\Bigl(2\varphi-\frac{(2\varphi)^3}{3!}+\frac{(2\varphi)^5}{5!}\Bigr)
-\frac12 \varphi^2.
\end{align*}
Now we use $\log(1+x)>\frac{x}{1+x/2}$ and simplify, yielding
\begin{equation}
h(\varphi)>\frac{\varphi^2 p(\varphi)}
{23040+1440\varphi^2-30\varphi^4}
\end{equation}
where $p(\varphi)$ is the polynomial
\begin{multline}
p(\varphi):=4\varphi^{8}+(20\log2-212)\varphi^6
-(1020\log2+1947)\varphi^4+\\+(7380-12480\log2)\varphi^2+(46080\log2-31680)
\end{multline}

The real roots of $p(\varphi)$ are $\pm\, 0.392976\dots$ and $\pm\,
7.78294\dots$, and $p(\varphi)$ is positive for
$0<\varphi<0.392976\dots$. The denominator  has only two real roots
at $\pm 7.78849\dots$.

So, $h(\varphi)>0$ for $0<\varphi<0.392976$, in particular for
$0<\varphi\le\frac13$.
\end{proof}

\begin{lemma}
$h(\varphi)>0$ for $3\le\varphi<\pi$.
\end{lemma}

\begin{proof}
For $3\le\varphi<\pi$ we have
\begin{gather*}
0<2\cos\frac{\varphi}{2}<2\cos\frac32,\quad
\log\Bigl(2\cos\frac\varphi2\Bigr)<\log\Bigl(2\cos\frac32\Bigr)<0,
\\
0<\cos6<\cos2\varphi<1,\quad \sin6<\sin2\varphi<0,\quad
-1<\cos\varphi<\cos3<0
\end{gather*}
so that
\begin{multline*}
h(\varphi)=\log2-1+\cos\varphi-\frac{\varphi}{2}\sin(2\varphi)
-\cos(2\varphi)\log\Bigl(2\cos\frac{\varphi}{2}\Bigr)>\\
>\log 2-1-1-\cos6\times \log\Bigl(2\cos\frac32\Bigr)\approx0.570891.
\end{multline*}
\end{proof}

Now we can apply the MSP to $h(\varphi)$ on the interval
$\frac{1}{3}\le \varphi\le 3$. First, we have to determine the
maximal slope of  $h(\varphi)$  on this interval.

\begin{lemma}
For $\frac{1}{3}\le\varphi\le3$ we have $|h'(\varphi)|<20 $.
\end{lemma}
\begin{proof}
We have $h(\varphi)= u(0)-u(\varphi)=u(0)+\Re
f(-e^{i\varphi})$. If we put $z=e^{i\varphi}$, then
$\frac{d}{d\varphi}=iz\frac{d}{dz}$. Hence
\begin{multline*}
h'(\varphi)=\Re
\Bigl\{iz\frac{d}{dz}\Bigl(-\frac{\log(1+z)}{z^2}+\frac{1}{z}-\frac{1}{2}\Bigr)\Bigr\}
=\\
=\Re\,
i\Bigl(\frac{2\log(1+z)}{z^2}-\frac{1}{z(1+z)}-\frac{1}{z}\Bigr).
\end{multline*}

It follows that for $\frac13\le\varphi\le3$ we have
\begin{multline}
|h'(\varphi)|\le 2|\log(1+z)|+\frac{1}{|1+z|}+1\le\\
\le
2(\bigl|\,\log|1+e^{3i}|\bigr|+\pi\,)
+\frac{1}{|1+e^{3i}|}+1<11+8+1=20.
\end{multline}
\end{proof}

Now applying the MSP (repeatedly) on the interval $[1/3, 3]$ we find
($\,$in 4163 steps$\,$) that indeed $h(\varphi)>0$ on this interval. The
procedure can be speeded up considerably by introducing a more
flexible  $M = M(t)$.
\begin{figure}[H]
  \includegraphics[width=\hsize]{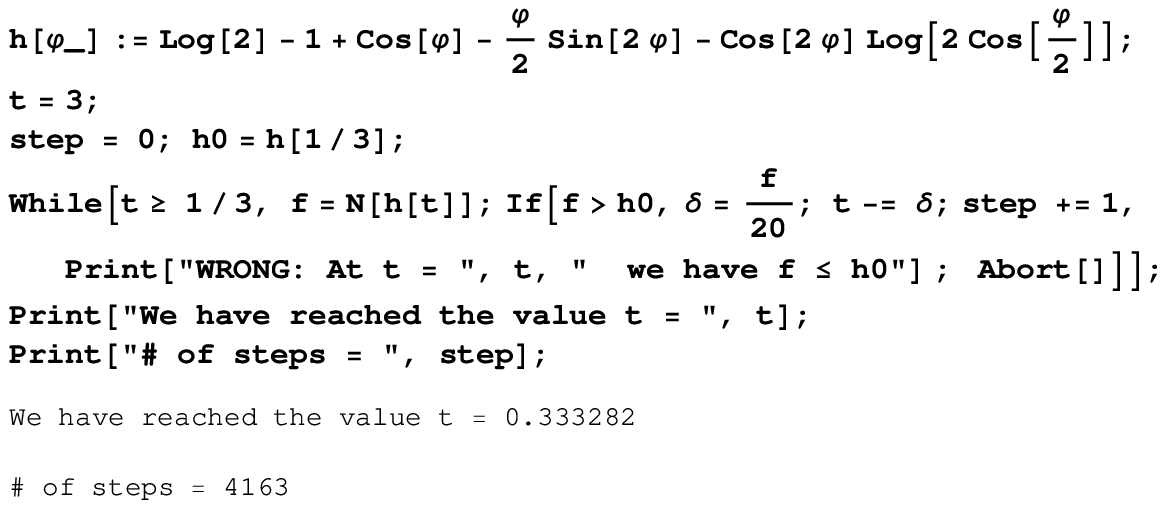}\\
\end{figure}
\noindent The other cases $r=\frac89$ and $r=0.883$ may be dealt with in a
similar manner.

Note: The above program is only an indication, for a complete
proof we must study the errors in the computations.  In the
computer all numbers are dyadic. So, what we need is a sequence
of dyadic numbers  $b=t_1>t_2>\cdots >t_m$   (without loss of
generality we may assume that  $b$  is dyadic) such that
$t_{k+1}>t_k-h(t_k)/M$ , for   $k=1$,   $2$,  \dots    $m-1$, with
$h(t_k)>0$   for all  $1\le k \le m$  and such that  $t_m < a$. In
our case a  more careful program will reveal that in the same
number of steps ($4163$) we get a   $t_m<1/3$, so that essentially
the above computation is correct.

\section{Once again the case $r=1$: Our Eulerian\\ approach.}

We will now show that $h(\varphi)$ as defined in \eqref{defh} is
strictly convex for $0<\varphi<\pi$.

Since $h'(0)=0$ this will solve our problem for $r=1$.

In view of the power series for $\log(1+z)$ our inequality may also
be written in the following interesting way
\begin{multline}\label{elementaryinequality}
u(\varphi)=\frac{\cos\varphi}{3}-\frac{\cos2\varphi}{4}+\frac{\cos3\varphi}{5}-+\cdots<
\frac{1}{3}-\frac{1}{4}+\frac{1}{5}-+\cdots=\log2-\frac12,\\
0<\varphi<\pi.
\end{multline}

Now we present a heuristic approach ---a technique often used
by Euler himself.

We write the left hand side of \eqref{elementaryinequality} as
\begin{equation}
u(\varphi)=\sum_{n=1}^\infty (-1)^{n-1}\frac{\cos(n\varphi)}{n+2}.
\end{equation}
Differentiating we find
\begin{multline*}
u''(\varphi)=\sum_{n=1}^\infty (-1)^n\frac{n^2\cos
n\varphi}{n+2}=\\= \sum_{n=1}^\infty (-1)^n\frac{n^2-4}{n+2}\cos
n\varphi+ \sum_{n=1}^\infty (-1)^n\frac{4}{n+2}\cos n\varphi=\\=
\sum_{n=1}^\infty (-1)^n(n-2)\cos n\varphi-4\sum_{n=1}^\infty
(-1)^n\frac{1}{n+2}\cos n\varphi.
\end{multline*}
Hence
\begin{multline*}
u''(\varphi)+4u(\varphi)=\sum_{n=1}^\infty (-1)^n(n-2)\cos
n\varphi=\\= \sum_{n=1}^\infty (-1)^n n\cos n\varphi+
2\sum_{n=1}^\infty (-1)^{n-1}\cos n\varphi
\end{multline*}
( see Hardy \cite[Section1.2~p.~2]{H} )
\begin{displaymath}
=-\frac12\frac{d}{d\varphi}\tan\frac{\varphi}{2}+2\times\frac12=
1-\frac12\frac{d}{d\varphi}\tan\frac{\varphi}{2}
\end{displaymath}
so that
\begin{equation}\label{difeqtan}
u''(\varphi)+4u(\varphi)=1-\frac12\frac{d}{d\varphi}\tan\frac{\varphi}{2}
\end{equation}
which may also be written as
\begin{equation}\label{diffeq}
u''(\varphi)+4u(\varphi)=1-\frac12\frac{1}{1+\cos \varphi}.
\end{equation}

Fully independent of the above Eulerian deduction, one may prove (
by direct verification ) that this differential equation for $u(t)$
is valid indeed.

\begin{proposition}
The function $u$ defined in \eqref{defu} satisfies the
differential equation \eqref{diffeq}.
\end{proposition}

\begin{proof}
Since $u(\varphi)$ is even we have
\begin{equation}
u(\varphi)=\Re\bigl(e^{-2i\varphi}\log(1+e^{i\varphi})-e^{-i\varphi}+\medio\bigr)=
\Re\Bigl(\frac{\log(1+z)}{z^2}-\frac{1}{z}+\frac12\Bigr)
\end{equation}
where $z=e^{i\varphi}$.  Then we have
$\frac{d}{d\varphi}=iz\frac{d}{dz}$. In this way we easily get
\begin{align*}
u'(\varphi)&=\Re\Bigl(\frac{i}{z}+\frac{i}{z(1+z)}-\frac{2i\log(1+z)}{z^2}\Bigr),\\
u''(\varphi)&=\Re\Bigl(\frac{1}{z}+\frac{1}{(1+z)^2}+\frac{3}{z(1+z)}-
\frac{4\log(1+z)}{z^2}\Bigr)
\end{align*}
so that
\begin{equation}
u''(\varphi)+4u(\varphi)=\Re\Bigl(\frac{z(1+2z)}{(1+z)^2}\Bigr).
\end{equation}
One may verify that
\begin{multline*}
\frac{z(1+2z)}{(1+z)^2}=\frac{z}{1+z}+\Bigl(\frac{z}{1+z}\Bigr)^2=
\frac{e^{i\varphi/2}}{e^{i\varphi/2}+e^{-i\varphi/2}}+\Bigl(\frac{e^{i\varphi/2}}{
e^{i\varphi/2}+e^{-i\varphi/2}}\Bigr)^2=\\=\frac{e^{i\varphi/2}}{2\cos\varphi/2}+
\frac{e^{i\varphi}}{4\cos^2\varphi/2}.
\end{multline*}
Taking real parts we get
\begin{equation}
u''(\varphi)+4u(\varphi)=\frac{1}{2}+\frac{\cos\varphi}{4\cos^2\varphi/2}=
\frac12+\frac{\cos\varphi}{2(1+\cos\varphi)}=1-\frac{1}{2(1+\cos\varphi)}.
\end{equation}
We  also have \eqref{difeqtan}. In fact
\begin{displaymath}
1-\frac12\frac{d}{d\varphi}\tan\frac{\varphi}{2}=1-\frac14\frac{1}{\cos^2\varphi/2}=1-\frac{1}{2(1+\cos\varphi)}
\end{displaymath}
\end{proof}

\begin{proposition}
The function $h$ may be represented by a power series
\begin{equation}\label{hseries}
h(\varphi)=\sum_{k=1}^\infty \frac{d_k}{(2k)!}(2\varphi)^{2k}
\end{equation}
where for $k\ge1$
\begin{equation}\label{formulad}
d_k=(-1)^k \frac{3}{4}-(-1)^k  \log2+(-1)^{k+1}
\sum_{j=1}^k\Bigl(1-\frac{1}{2^{2j}}\Bigr)\frac{B_{2j}}{2j}.
\end{equation}
\end{proposition}

\begin{proof}
By \eqref{formulah} we know that $h$ is analytic for
$|\varphi|<\pi$, so that  \eqref{hseries} is valid for
$|\varphi|<\pi$. To determine the coefficients, observe that,
because $h(\varphi)=u(0)-u(\varphi)$ by \eqref{difeqtan}, we
have
\begin{equation}
h''(\varphi)+4h(\varphi)=4u(0)-1+\frac12\frac{d}{d\varphi}\tan\frac{\varphi}{2}
\end{equation}
so that
\begin{multline*}
4\sum_{k=1}^\infty2k(2k-1)\frac{d_k}{(2k)!}(2\varphi)^{2k-2}
+4\sum_{k=1}^\infty
\frac{d_k}{(2k)!}(2\varphi)^{2k}=\\
=4\log2-3+\sum_{k=1}^\infty (-1)^{k-1}
\frac{(2k-1)(2^{2k}-1)B_{2k}}{(2k)!}\varphi^{2k-2}.
\end{multline*}
Equating  coefficients of equal powers of $\varphi$ we get
\begin{equation}
4d_1=4\log2-3+\frac{3}{2}B_2=4\log2 - \frac{11}{4}
\end{equation}
and for $k\ge1$
\begin{equation}
4(2k+2)(2k+1)\frac{2^{2k}d_{k+1}}{(2k+2)!}+4\frac{2^{2k}d_k}{(2k)!}=(-1)^{k}
\frac{(2k+1)(2^{2k+2}-1)B_{2k+2}}{(2k+2)!}
\end{equation}
which simplifies to
\begin{equation}
d_{k+1}=-d_k+(-1)^k (1-2^{-2k-2})\frac{B_{2k+2}}{2k+2}, \qquad
k\ge2.
\end{equation}
Now we can prove formula \eqref{formulad} by induction. First,
for $k=1$, \eqref{formulad} gives the correct value of $d_1$.
Assuming that \eqref{formulad} is true for $k$ we get
\begin{align*}
d_{k+1}&=-d_k+(-1)^k (1-2^{-2k-2})\frac{B_{2k+2}}{2k+2}= \\&=-(-1)^k
\frac{3}{4}+(-1)^k  \log2-(-1)^{k+1}
\sum_{j=1}^k\Bigl(1-\frac{1}{2^{2j}}\Bigr)\frac{B_{2j}}{2j}+\\
&\qquad+(-1)^k (1-2^{-2k-2})\frac{B_{2k+2}}{2k+2}\\
&=(-1)^{k+1}\frac{3}{4}-(-1)^{k+1}  \log2+(-1)^{k+2}
\sum_{j=1}^{k+1}\Bigl(1-\frac{1}{2^{2j}}\Bigr)\frac{B_{2j}}{2j}
\end{align*}
so that \eqref{formulad} is also true for $d_{k+1}$.
\end{proof}

\begin{proposition}
All coefficients $d_k$ in the Taylor expansion \eqref{hseries} are
strictly positive.
\end{proposition}

\begin{proof}
Recall the well known formula \cite{E}
\begin{equation}
\zeta(2n)=(-1)^{n+1}\frac{(2\pi)^{2n} B_{2n}}{2\cdot (2n)!}
\end{equation}
so that  we may write \eqref{formulad} as
\begin{equation}
d_k=(-1)^{k+1}
\sum_{j=1}^k(-1)^{j+1}\Bigl(1-\frac{1}{2^{2j}}\Bigr)\frac{2\cdot(2j-1)!\zeta(2j)}{(2\pi)^{2j}}+(-1)^k
\frac{3}{4}-(-1)^k  \log2.
\end{equation}
Therefore, since $\log2<3/4$,
\begin{multline*}
d_k\ge\\
\Bigl(1-\frac{1}{2^{2k}}\Bigr)\frac{2\cdot(2k-1)!\zeta(2k)}{(2\pi)^{2k}}-
\sum_{j=1}^{k-1}\Bigl(1-\frac{1}{2^{2j}}\Bigr)\frac{2\cdot(2j-1)!\zeta(2j)}{(2\pi)^{2j}}-
\Bigl(\frac{3}{4}-\log2\Bigr).
\end{multline*}
So, we only need to prove that
\begin{equation}\label{needed}
\sum_{j=1}^{k-1}\Bigl(1-\frac{1}{2^{2j}}\Bigr)\frac{2\cdot(2j-1)!\zeta(2j)}{(2\pi)^{2j}}+
\Bigl(\frac{3}{4}-\log2\Bigr)<\Bigl(1-\frac{1}{2^{2k}}\Bigr)
\frac{2\cdot(2k-1)!\zeta(2k)}{(2\pi)^{2k}}.
\end{equation}
But we have
\begin{equation}\label{parcial}
\sum_{j=1}^{k-1}\Bigl(1-\frac{1}{2^{2j}}\Bigr)\frac{2\cdot(2j-1)!\zeta(2j)}{(2\pi)^{2j}}
\le 2\zeta(2)\sum_{j=1}^{k-1}\frac{(2j-1)!}{(2\pi)^{2j}}.
\end{equation}
For $k\ge8$ the last term of the sum in the right hand side of
\eqref{parcial} is the greatest,  so that
\begin{displaymath}
2\zeta(2)\sum_{j=1}^{k-1}\frac{(2j-1)!}{(2\pi)^{2j}}\le 2
(k-1)\zeta(2)\frac{(2k-3)!}{(2\pi)^{2k-2}}=\zeta(2)\frac{(2k-2)!}{(2\pi)^{2k-2}},\qquad
k\ge8.
\end{displaymath}
Also, it is easy to check that
\begin{equation}
\zeta(2)\frac{(2k-2)!}{(2\pi)^{2k-2}}>\Bigl(\frac34-\log2\Bigr),\qquad
(k\ge8).
\end{equation}
It follows that for $k\ge8$  inequality \eqref{needed} would be a
consequence of
\begin{equation}
2\zeta(2)\frac{(2k-2)!}{(2\pi)^{2k-2}}<
\Bigl(1-\frac{1}{2^{2k}}\Bigr)\frac{2\cdot(2k-1)!\zeta(2k)}{(2\pi)^{2k}}.
\end{equation}
This follows from the inequality
\begin{equation}
2\zeta(2)\frac{(2k-2)!}{(2\pi)^{2k-2}}<\frac{2^{16}-1}{2^{16}}\cdot\frac{2\cdot(2k-1)!}{(2\pi)^{2k}}.
\end{equation}
So, we only need to show that
\begin{equation}
\frac{2^{16}}{2^{16}-1}(2\pi)^2\zeta(2)< (2k-1)
\end{equation}
which is true for $k\ge33$.

It remains  to prove that $d_k>0$ for $1\le k\le 32$. Each of the
numbers $d_k$ is of the form $\frac{a}{b}\pm \log 2$. Each
inequality $d_k>0$ can be written as  $\log 2>r$ or $\log 2<s$ where
$r$ and $s$ are certain rational numbers. It is easy to see that
$\max r = \frac{177}{256}$ and $\min s = \frac{89}{128}$ and we
check that in fact
\begin{displaymath}
0.691406\approx\frac{177}{256}<\log2\approx0.693147<\frac{89}{128}\approx0.695313
\end{displaymath}
finishing the proof that $d_k>0$ for all $k\ge1$.
\end{proof}

\section{The general case.}

For  $0<r<1$ we want to prove that $-\Re f(re^{i\varphi})\le -\Re
f(-r)$ for $-\pi<\varphi\le \pi$. As before we change variables
putting $\pi-\varphi$ instead of $\varphi$. So, we want to prove
that $-\Re f(-re^{-i\varphi})\le -\Re f(-r)$.

Because $-\Re f(-re^{-i\varphi})=-\Re f(-re^{i\varphi})$ we
will show that
\begin{equation}
-\Re f(-re^{i\varphi})\le-\Re f(-r),\qquad -\pi<\varphi\le\pi.
\end{equation}
For $|z|<1$ we  define
\begin{multline}
U(z)=U(r e^{i\varphi}):=\Re f(-re^{i\varphi})-\Re f(-1)=\\
=-\Re
f(-1)-\Re\Bigl(
\frac{\log(1+z)}{z^2}-\frac{1}{z}+\frac{1}{2}\Bigr),\qquad
z=re^{i\varphi}.
\end{multline}
Then $U$ is a harmonic function on the unit disc $\Delta:=\{z:
|z|<1\}$. In fact it extends to a continuous function on
$\Delta\smallsetminus\{-1\}$. This extension will also be denoted by
$U$. The   values of $U(e^{i\varphi})$  at  the boundary of $\Delta$
coincide with those of  $h(\varphi)$ as defined by \eqref{defh}. Our
problem is to show that for $0<r<1$ and $-\pi<\varphi<\pi$ we have
$U(r e^{i\varphi})\ge U(r)$.

Because $h(\varphi)$ is an $\mathcal{L}^1(0,2\pi)$ function  we
have
\begin{equation}\label{Poisson}
U(re^{i\varphi})=\frac{1}{2\pi}\int_0^{2\pi}
h(t)P_r(\varphi-t)\,dt,\qquad 0<r<1,\quad -\pi<\varphi<\pi
\end{equation}
where $P_r(t):=\frac{1-r^2}{1+r^2-2r\cos t}$ is the Poisson kernel.

Our claim will now follow from some ( slightly adapted )
theorems on rearrangements as described in the book by
Hardy-Littlewood-Pólya  on inequalities \cite[Theorems 368 and
378]{HLP}. Since the theorems there  do not apply directly to
our situation we prove the following:
\begin{proposition}\label{PropReargm}
Let $F$ and $G$ be   measurable positive periodic functions on $\R$,
with period $2\pi$. We assume that $F$ and $G$ are even, and that
$F$ is non decreasing and $G$ non increasing on $(0,\pi)$.

If $T\colon(0,2\pi]\to(0,2\pi]$ is a Borel measurable function that
preserves Lebesgue measure, i.~e. for any Borel $B\subset(0, 2\pi]$
we have $|T^{-1}(B)|=|B|$, then
\begin{equation}\label{rearrangementinequality}
\int_0^{2\pi} F(t) G(t)\,dt\le \int_0^{2\pi} F(t) G(T(t))\,dt.
\end{equation}
\end{proposition}

\begin{proof}
Consider first the case in which $F$ and $G$ only take the
values $0$ and $1$. Then, the hypotheses of the Proposition
imply  that $G$ is the characteristic function of an interval
$I$ with center at $0$ and $F$ the characteristic function of
an interval $J$ with center at $\pi$ ( considering the
functions $F$ and $G$ as defined on the circle ( group ) ).
Then
\begin{displaymath}
|I\cap J|=\int_0^{2\pi} F(t)G(t)\,dt\quad\text{and}\quad |I\cap M|=
\int_0^{2\pi} F(t)G(T(t))\,dt
\end{displaymath}
where $M=T^{-1}(J)$ is a measurable set of measure $|M|=|J|$. If
$I\cap J=\emptyset $ there is nothing to prove. In the other case we
will have
\begin{displaymath}
|I|+|J|-|I\cap J|=|I\cup J|=2\pi\quad\text{and}\quad |I|+|M|-|I\cap
M|=|I\cup M|\le2\pi
\end{displaymath}
and it follows that $|I\cup J|\le|I\cap M|$.

In the general case  $F$ and $G$  can be written as the suprema  of
increasing sequences of step functions of type $F=\lim F_r$, with
$F_r:=\sum_{k=1}^n a_k \chi_{J_k}$ and $G=\lim G_r$ with $G_r:=
\sum_{k=1}^m b_k\chi_{I_k}$, where $a_k\ge0$, $b_k\ge0$, the $J_k$
are intervals centered at $\pi$ and the $I_k$ intervals centered at
$0$.

Then the result for intervals implies
\begin{displaymath}
\int_0^{2\pi} F_r(t)G_r(t)\,dt\le \int_0^{2\pi}
F_r(t)G_r(T(t))\,dt.
\end{displaymath}
Applying the Monotone
Convergence Theorem we get \eqref{rearrangementinequality}.
\end{proof}

\begin{theorem}
For  $0<r<1$ and $0<\varphi<\pi$ we have $-\Re f(re^{i\varphi})<-\Re
f(-r)$.
\end{theorem}

\begin{proof}
The inequality is equivalent to
\begin{displaymath}
\Re f(-r)-\Re f(-1) =U(r)   < U(r e^{i\varphi})= \Re
f(-re^{i\varphi})-\Re f(-1)
\end{displaymath}
We can apply Proposition \ref{PropReargm} to the representation
\eqref{Poisson}. In fact our $h(t)$ is even, positive  and non
decreasing on $(0,\pi)$, and the Poisson kernel
$P_r(t)=\frac{1-r^2}{1+r^2-2r\cos t}$ is even, positive and non
increasing on $(0,\pi)$. Also the translation $t\mapsto\varphi-t$ is
measure preserving on the circle. So Proposition \ref{PropReargm}
yields $U(r)\le U(re^{i\varphi})$.

To show that the inequality is strict for $0<\varphi<\pi$, we
consider a small $\delta>0$ such that
$0<a:=\varphi/2-\delta<\varphi/2<b:=\varphi/2+\delta<\pi$, and
also a small $\varepsilon >0$ such that
$0<a-\varepsilon<a+\varepsilon<\varphi/2<b-\varepsilon<b+\varepsilon<\pi$.
Consider the intervals $I_a:=[a-\varepsilon,a+\varepsilon]$ and
$I_b:=[b-\varepsilon,b+\varepsilon]$.  The transformation
$t\mapsto \varphi-t$ transforms $I_a$ into $I_b$ and $I_b$ into
$I_a$. Now consider the transformation $T$ such that
$T(t)=\varphi-t$ when $t\notin I_a\cup I_b$. For $t\in I_a$ we
define $T(t)=2a-t$ and for $t\in I_b$ we put $T(t)=2b-t$ ( $t$
and $2b-t$ are symmetrical with respect to $b$ ). It is clear
that $T$ conserves the measure of $(-\pi,\pi]$ ( considered as
the circle ).  We will prove that
\begin{equation}\label{end}
\int_{-\pi}^{\pi} h(t)P_r(\varphi-t)\,dt\le \int_{-\pi}^{\pi}
h(t)P_r(T(t))\,dt<\int_{-\pi}^{\pi}
h(t)P_r(\varphi-t)\,dt
\end{equation}
thereby concluding the proof.

The first inequality is simply a new application of Proposition
\ref{PropReargm}. We only need to confirm the second inequality
in \eqref{end}.  By definition $T(t)=\varphi-t$ except on
$I_a\cup I_b$ so that
\begin{multline}
D:=\int_{-\pi}^{\pi} h(t)P_r(\varphi-t)\,dt-\int_{-\pi}^{\pi}
h(t)P_r(T(t))\,dt=\\=\int_{I_a\cup I_b}h(t)P_r(\varphi-t)\,dt-
\int_{I_a\cup I_b} h(t)P_r(T(t))\,dt=
\int_{a-\varepsilon}^{a+\varepsilon}h(t)P_r(\varphi-t)\,dt+\\+
\int_{b-\varepsilon}^{b+\varepsilon}h(t)P_r(\varphi-t)\,dt-
\int_{a-\varepsilon}^{a+\varepsilon}h(t)P_r(T(t))\,dt-
\int_{b-\varepsilon}^{b+\varepsilon}h(t)P_r(T(t))\,dt.
\end{multline}
Now we change variables so that all integrals are taken over the
same interval $(-\varepsilon,\varepsilon)$. Observing that
$a=\varphi/2-\delta$ and $\varphi-a=b$
\begin{multline}
D=\int_{-\varepsilon}^{\varepsilon}h(a+t)P_r(b-t)\,dt+
\int_{-\varepsilon}^{\varepsilon}h(b+t)P_r(a-t)\,dt-\\-
\int_{-\varepsilon}^{\varepsilon}h(a+t)P_r(a-t)\,dt-
\int_{-\varepsilon}^{\varepsilon}h(b+t)P_r(b-t)\,dt
\end{multline}
we find that
\begin{equation}
D=\int_{-\varepsilon}^{\varepsilon}\bigl(h(b+t)-h(a+t)\bigr)\bigl(P_r(a-t)-P_r(b-t)
\bigr)\,dt.
\end{equation}
Here we always have   $0<a+t<b+t<\pi$, and $0<a-t<b-t<\pi$ so that
the integrand is strictly positive. We thus have $D>0$, completing
the proof.
\end{proof}
\bigskip

\textsc{Acknowledgement: } The authors would like to thank
Foster Dieckhoff ($\,$Kansas City, Missouri, USA$\,$) and Maarten van
Swaay ($\,$for\-mer\-ly at Kansas State University, Manhattan, Kansas,
USA$\,$) for their linguistic assistance in preparing this note,
and for their interest in the subject.


\begin{thebibliography}{}

\bibitem{A}
\textsc{J.~Arias de Reyna}, \emph{High Precision Computation of
Riemann's Zeta Function by the Riemann-Siegel Formula, I}, to
appear.

\bibitem{E}
\textsc{H.~M.~Edwards}, \emph{Riemann's zeta
    function}. Reprint of the 1974 original [Academic Press,
    New York],  Dover Publications Inc., Mineola, NY, 2001.


\bibitem{FJ}
\textsc{M. Fitzgerald \&\ I. James}, \emph{The mind of the
Mathematician}, The Johns Hopkins University Press, Baltimore, 2007,

\bibitem{G} \textsc{W.~Gabcke}, \emph{Neue Herleitung und
    explizite Restabschätzung der Riemann-Siegel-Formel},
    Dissertation, Göttingen (1979).

\bibitem{H} \textsc{G.~Hardy}, \emph{Divergent Series}, Oxford
    University Press, 1949.

\bibitem{HLP}
\textsc{G.~Hardy, J.~E.~Littlewood \&\ G.~Pólya},
    \emph{Inequalities},  Second Ed., Cambridge University Press,
    Cambridge, 1991.

\end{thebibliography}
\end{document}